\documentclass[12pt]{amsart}
\usepackage{amscd}
\date{}

\def\NZQ{\Bbb}               
\def\NN{{\NZQ N}}

\def\ZZ{{\NZQ Z}}

\def\frk{\mathfrak}               

\def\mm{{\frk m}}

\def\a{{\bold a}}

\def\e{{\bold e}}
\def\opn#1#2{\def#1{\operatorname{#2}}} 
\opn\chara{char} \opn\length{\ell} \opn\pd{pd} \opn\rk{rk}
\opn\projdim{proj\,dim} \opn\injdim{inj\,dim} \opn\rank{rank}
\opn\depth{depth} \opn\codepth{codepth} \opn\grade{grade}
\opn\height{ht} \opn\embdim{emb\,dim} \opn\codim{codim} \opn\ara{ara}

\opn\Tr{Tr} \opn\bigrank{big\,rank}
\opn\superheight{superheight}\opn\lcm{lcm}
\opn\trdeg{tr\,deg}
\opn\reg{reg} \opn\lreg{lreg} \opn\skel{skel} \opn\Gr{Gr}
\opn\dim{dim} \opn\arithdeg{arithdeg}
\opn\Spec{Spec} \opn\Supp{Supp} \opn\supp{supp} \opn\Sing{Sing}
\opn\Ass{Ass}
\opn\Ann{Ann} \opn\Rad{Rad} \opn\Soc{Soc}
\opn\Sym{Sym} \opn\Ker{Ker} \opn\Coker{Coker} \opn\Im{Im}
\opn\Hom{Hom} \opn\Tor{Tor} \opn\Ext{Ext} \opn\End{End}
\opn\Aut{Aut} \opn\id{id} \opn\ini{in}
\opn\aff{aff} \opn\con{conv} \opn\relint{relint} \opn\st{st}
\opn\lk{lk} \opn\cn{cn} \opn\core{core} \opn\vol{vol}
\opn\link{lk} 
\opn\star{st} 
\opn\skel{skel}
\opn\gr{gr}

\def\pot#1#2{#1[\kern-0.28ex[#2]\kern-0.28ex]}

\opn\dirlim{\underrightarrow{\lim}}
\opn\inivlim{\underleftarrow{\lim}}
\let\to=\rightarrow

\newtheorem{Theorem}{Theorem}[section]
\newtheorem{Lemma}[Theorem]{Lemma}
\newtheorem{Corollary}[Theorem]{Corollary}
\newtheorem{Proposition}[Theorem]{Proposition}
\newtheorem{Remark}[Theorem]{Remark}

\newtheorem{Example}[Theorem]{Example}

\let\epsilon\varepsilon
\let\phi=\varphi
\let\kappa=\varkappa
\opn\Gin{Gin}

\def\D{\Delta}
\def\G{\Gamma}
\def\F{{\mathcal F}}
\def\0{\bold 0}

\opn\inii{in} \opn\inim{inm} \opn\rate{rate}

\numberwithin{equation}{section}

\begin{document}

\title[powers of Stanley-Reisner ideals]{Cohen-Macaulayness of large powers of Stanley-Reisner ideals}
\author{Naoki Terai}
\medskip
\address{Department of Mathematics, Faculty of Culture
and Education, Saga University, Saga 840--8502, Japan.} 
\email{terai@cc.saga-u.ac.jp}
\author{Ngo Viet Trung }
\address{Institute of Mathematics, 18 Hoang Quoc Viet, Hanoi,Vietnam} 
\email{nvtrung@math.ac.vn}
 
\thanks{The first author is supported by JSPS 20540047 and the second author by the National Foundation of  Science and Technology Development.}

\begin{abstract}
We prove that for $m \ge 3$, the symbolic power $I_\D^{(m)}$ of a Stanley-Reisner ideal is Cohen-Macaulay if and only if the simplicial complex $\D$ is a matroid. Similarly, the ordinary power $I_\D^m$ is Cohen-Macaulay for some $m \ge 3$ if and only if $I_\D$ is a complete intersection. These results  solve several open questions on the Cohen-Macaulayness  of ordinary and symbolic powers of Stanley-Reisner ideals. Moreover, they have interesting consequences on the Cohen-Macaulayness  of symbolic powers of  facet ideals and cover ideals. 
\end{abstract}

\maketitle

\section{Introduction}

In this article we study the Cohen-Macaulayness of
ordinary and  symbolic powers of squarefree ideals.
More concretely, we consider such an ideal as the Stanley-Reisner ideal $I_\D$ of a simplicial complex $\D$ and we give conditions on $\D$ such that the ordinary or symbolic powers $I_\D^m$ or $I_\D^{(m)}$ is Cohen-Macaulay for a fixed number $m \ge 3$.
\smallskip

A classical theorem \cite{CN} implies that all ordinary powers $I_\D^m$ are Cohen-Macaulay if and only if $I_\D$ is a complete intersection. On the other hand, there are examples such that the second power $I_\D^2$ is Cohen-Macaulay but the third power $I_\D^3$ is not Cohen-Macaulay. It does not deny any logical possibility that there may exist  a Stanley-Reisner ideal such that its 100-th power is Cohen-Macaulay but  its 101-th power is not. Hence it is of great interest to find combinatorial conditions on $\D$ which guarantees the Cohen-Macaulayness of a given ordinary power of $I_\D$. Since $I_\D^m$ is Cohen-Macaulay if and only if  $I_\D^{(m)}$ is Cohen-Macaulay and $I_\D^{(m)} = I_\D^m$, one may raise the same problem on the Cohen-Macaulayness of a given symbolic power of  $I_\D$. \smallskip

Studies in this direction  began in \cite{CRTY,MiT1,RTY1} for the cases
$\dim \D = 1$  ($\D$ is a graph) or $\D$ is a flag complex ($I_\D$ is generated by quadratic monomials), where one could give precise combinatorial conditions for the Cohen-Macaulayness of each power $I_\D^m$ or $I_\D^{(m)}$ in terms of $\D$. 
For the general case,  one was only able to give conditions for the Cohen-Macaulayness of $I_\D^{(2)}$ and $I_\D^2$ \cite{MiT2,RTY2}. To characterize the Cohen-Macaulayness of $I_\D^{(m)}$ and $I_\D^m$ seems to be a difficult problem. In fact, there were open questions which would have an answer if one could solve this problem. 
\smallskip

First of all, results on the preservation of Cohen-Macaulayness among large powers of $I_\D$ \cite{MiT2} led to the following
\medskip

\noindent{\bf Question 1}. 
Is $I_{\Delta }^{(m)}$ is Cohen-Macaulay if $I_{\Delta }^{(m+1)}$ is Cohen-Macaulay?
\medskip

\noindent{\bf Question 2}.   
Does there exist a number $t$ depending on $\dim \Delta$ such that 
if $I_{\Delta }^{(t)}$ is Cohen-Macaulay, then  
$I_{\Delta }^{(m)}$ is Cohen-Macaulay for every $m \ge 1$?
\medskip

If $\dim \D = 2$, there is a complete classification of complexes possessing a Cohen-Macaulay ordinary power in \cite{TrTu}. Together with results for $\dim \D=1$ \cite{MiT1} and for flag complexes \cite{RTY1} this suggests the following
\medskip

\noindent{\bf Question 3}.   
Is $I_{\Delta }^{m}$ Cohen-Macaulay for some  $m \ge 3$ if and only if $I_\D$ is a complete intersection?
\medskip

On the other hand, works on the Buchsbaumness or on Serre condition $(S_2)$ of ordinary and symbolic powers of Stanley-Reisner ideals in \cite{MN1, MN2, RTY1, TY} showed that these properties are strongly related to the Cohen-Macaulayness and that one may raise the following
\medskip

\noindent{\bf Question 4}.  
Is $I_{\Delta }^{(m)}$ (or $I_{\Delta }^{m}$) Cohen-Macaulay if  $I_{\Delta }^{(m)}$ (or $I_{\Delta }^{m}$) is (quasi-)Buchsbaum or satisfies $(S_2)$ for large $m$?
\medskip

This paper will give a positive answer to all these questions. More precisely, we prove
the following characterizations of the Cohen-Macaulayness of  $I_\D^m$ and $I_\D^{(m)}$ for $m \ge 3$.

\begin{Theorem}\label{SymMain}
Let $\D$ be a simplicial complex with $\dim \D \ge 2$. Then the following conditions are equivalent: \par
{\rm (i)} $I_{\Delta }^{(m)}$ is Cohen-Macaulay for every $m \ge 1$.\par
{\rm (ii)} $I_{\Delta }^{(m)}$ is Cohen-Macaulay for some $m \ge 3$.\par
{\rm (iii)} $I_{\Delta }^{(m)}$ satisfies $(S_2)$ for some $m \ge 3$.\par
{\rm (iv)} $I_{\Delta }^{(m)}$ is (quasi-)Buchsbaum for some $m \ge 3$.\par
{\rm (v)} $\Delta$ is a matroid.
\end{Theorem}

\begin{Theorem}\label{OrdMain}
Let $\D$ be a simplicial complex with $\dim \D \ge 2$. Then the following conditions are equivalent:\par
{\rm (i)} $I_{\Delta }^{m}$ is Cohen-Macaulay for every $m \ge 1$.\par
{\rm (ii)} $I_{\Delta }^{m}$ is Cohen-Macaulay for some $m \ge 3$.\par
{\rm (iii)} $I_{\Delta }^{m}$ satisfies $(S_2)$ for some $m \ge 3$.\par
{\rm (iv)} $I_{\Delta }^{m}$ is (quasi-)Buchsbaum for some $m \ge 3$.\par
{\rm (v)} $\Delta$ is a complete intersection.
\end{Theorem}

These theorems are remarkable in the sense that the Cohen-Macaulayness of $I_\D^{(m)}$ or $I_\D^m$ is equivalent to much weaker properties and that it can be characterized in purely combinatorial terms {\em which are the same for all $m \ge 3$}. 
Both theorems except condition (iii) also hold for $\dim \D = 1$ \cite{MiT1}. If $\dim \D = 1$, the Buchsbaumness behaves a little bit differently. Minh-Nakamura \cite{MN1} showed that for a graph $\D$, $I_\D^{(3)}$ is Buchsbaum if and only if $I_\D^{(2)}$ is Cohen-Macaulay and that for $m \ge 4$, $I_\D^{(m)}$  is Cohen-Macaulay if $I_\D^{(m)}$ is Buchsbaum.  Similar results also hold for the Buchsbaum property of $I_\D^m$ \cite{MN2}.
\smallskip

The equivalence (i) $\Leftrightarrow$ (v) of the first theorem was discovered in \cite{MiT2, Va}.
This result establishes an unexpected link between a purely algebraic property with a vast area of combinatorics. 
The equivalence (i) $\Leftrightarrow$ (vi) of the second theorem was proved 
by Cowsik-Nori \cite{CN} under much more general setting.
Our new contributions are (iii)$\Rightarrow$(v) and (iv)$\Rightarrow$(v) in both theorems.
The proofs involve both algebraic and combinatorial arguments.
\smallskip

The idea for the proof of (iii)$\Rightarrow$(v) comes from the fact that matroids and complete intersection complexes can be characterized by properties of their links. We call $\D$ locally a matroid or a complete intersection if the links of $\D$ at the vertices are matroids or complete intersections, respectively. It turns out that a complex is a matroid if and only if it is connected and locally a matroid (Theorem \ref{SymLocal}).  
A similar result on complete intersection was already proved in \cite{TY}. 
The connectedness of $\D$ can be studied by using Takayama's formula for the local cohomology modules of monomial ideals \cite{Ta}. Since the links of $\D$ correspond to localizations of $I_\D$, these results allow us to reduce our investigation to the one-dimensional case for which everything is known by \cite{MiT1}. \smallskip

The proof of (iv)$\Rightarrow$(v) follows from our investigation on the generalized Cohen-Macaulayness of $I_\D^{(m)}$ or $I_\D^m$. It is known that (quasi-)Buchsbaum rings are generalized Cohen-Macaulay. In general, it is difficult to classify generalized Cohen-Macaulay ideals because this class of ideals is too large. However, we can prove that $I_\D^{(m)}$  is  generalized Cohen-Macaulay for some $m \ge 3$ or for every $m \ge 1$  if and only if $\D$ is a union of disjoint matroids of the same dimension (Theorem \ref{SymFLC}). 
Similarly, $I_\D^m$  is  generalized Cohen-Macaulay for some $m \ge 3$ or for every $m \ge 1$  if and only if $\D$ is a union of disjoint complete intersections of the same dimension (Theorem \ref{OrdFLC}). 
A weaker version of this result was proved earlier by Goto-Takayama \cite{GT}. For flag complexes it was proved in \cite{RTY1}. \smallskip

As applications we study the Cohen-Macaulayness of symbolic powers of the facet ideal $I(\D)$ and the cover ideal $J(\D)$, which are generated by the squarefree monomials of the facets of $\D$ or of their covers, respectively. \smallskip

To study ideals generated by squarefree monomials of the same degree $r$ means to study facet ideals of pure complexes of dimension $r-1$. By \cite{RTY1} we know that for a pure complex  $\D$ with $\dim \D = 1$, $I(\D)^{(m)}$ is Cohen-Macaulay for some $m \ge 3$ or for every $m \ge 1$ if and only if $\D$ is a union of disjoint 1-uniform matroids, where a matroid is called  $r$-uniform if it is generated by the $r$-dimensional faces of a simplex.  This gives a structure theorem for squarefree monomial ideals generated in degree 2 whose symbolic powers are Cohen-Macaulay. It is natural to ask whether there are similar results for squarefree monomial ideals generated in degree $\ge 3$. We prove that for a pure complex  $\D$ with $\dim \D = 2$, $I(\D)^{(m)}$ is Cohen-Macaulay for some $m \ge 3$ or for every $m \ge 1$ if and only if $\D$ is a union of disjoint 2-uniform matroids (Theorem \ref{2-uniform}), and we show that a similar result doesn't hold for $\dim \D \ge 3$.
\smallskip

A cover ideal $J(\D)$ can be understood as the intersection of prime ideals generated by the variables of the facets of $\D$. It turns out that for $m \ge 3$, $J(\D)^{(m)}$ is Cohen-Macaulay if and only if $\D$ is a matroid (Theorem \ref{cover}) so that the Cohen-Macaulayness of $J(\D)^{(m)}$ is equivalent to that of $I_\D^{(m)}$. This result is somewhat surprising because the Cohen-Macaulayness of $I_\D$ usually has nothing to do with that of $J(\D)$. As a consequence we can say exactly when all symbolic powers of a squarefree monomial ideal of codimension 2 is Cohen-Macaulay. \smallskip

Summing up we can say that our results provide a framework for the study of the Cohen-Macaulayness and other ring-theoretic properties of large ordinary and symbolic powers of squarefree monomial ideals. \smallskip

The paper is organized as follows. In Section 2 we prepare basic facts and properties of simplicial complexes, especially of localizations and matroids. Section 3 and Section 4 deal with the Cohen-Macaulayness of symbolic and ordinary powers of Stanley-Reisner ideals. In Section 5 we study the Cohen-Macaulayness of symbolic powers of facet ideals and cover ideals. For unexplained terminology we refer to \cite{BH} and \cite{St}.

\section{Localizations and matroids}

Throughout this article let $K$ be an arbitrary field and $[n] = \{1,...,n\}$. 
Let $\D$ be a (simplicial) complex on the vertex set $V(\D) = [n]$.
The \textit{Stanley--Reisner ideal} $I_{\Delta}$ of  $\Delta$ (over $K$) is defined as the squarefree monomial ideal 
\[
 I_{\Delta}=\big(x_{i_1} x_{i_2} \cdots x_{i_p}|\ 1 \le i_1 < \cdots < i_p \le n,\; \{i_1,\ldots, i_p\} \notin \Delta \big) 
\]
in the polynomial ring $S = K[x_1,x_2,\dots , x_n]$. Note that the sets $\{i_1,\ldots, i_p\} \notin \D$ are called {\it nonfaces} of $\D$.
It is obvious that this association gives an one-to-one correspondence between simplicial complexes on the vertex set $[n]$ and squarefree monomial ideals in $S$ which do not contain any variable.
\smallskip

We will consider a graph as a complex, and we will always assume that a graph has no multiple edges, no loops and no isolated vertices. \smallskip

For every subset $F \subseteq [n]$, we set $P_F = (x_i|\ i \in F)$. Let $\F(\D)$ denote the set of the facets of $\D$. We have the following prime decomposition 
of $I_{\Delta}$:
\[
 I_{\Delta} = \bigcap_{F \in \mathcal{F}(\Delta)} P_{\overline F},
\]
where $\overline F$ denotes the complement of $F$. \smallskip

To describe the localization of $I_\D$ at a set of variables we need the following notations. 
For every face $G \in \Delta$, we define 
\begin{align*}
\star_{\Delta} G & := \{F \in \Delta \;:\, F \cup G \in \Delta\},\\
\link_{\Delta} G & := \{F \setminus G \in \Delta \;:\, G \subset F \in \Delta  \},
\end{align*}
and call these subcomplexes of $\D$ the $\textit{star}$ of $G$ or the \textit{link} of $G$, respectively. \smallskip

By the definition of the Stanley-Reisner ideal, $I_{\star_{\Delta} G}$ and  $I_{\link_{\Delta} G}$ have the same (minimal) monomial generators though they lie in different polynomial subrings of $S$ if $G \neq \emptyset$.  

\begin{Lemma} \label{local}
Let $Y = \{x_i|\ i \not \in V(\star_{\Delta}G)\}$. Then
$$I_{\Delta}S[x_i^{-1}|\ i \in G]  = (I_{\link_{\Delta}G},Y)S[x_i^{- 1}|\ i \in G].$$
\end{Lemma}

\begin{proof}
It is easily seen that
\begin{align*}
I_{\Delta}S[x_i^{-1}|\ i \in G] & = \bigcap_{F \in \mathcal{F}(\star_{\Delta}G)} P_{\bar F}S[x_i^{-1}|\ i \in G]\\
& = (I_{\star_{\Delta}G},Y)S[x_i^{-1}|\ i \in G]\\
& = (I_{\link_{\Delta}G},Y)S[x_i^{-1}|\ i \in G].
\end{align*}
\end{proof}

\begin{Remark} \label{link}
{\rm  Let $R = K[x_i|\ i \in V(\link_\D G)]$ and $T = K[x_i|\ x_i \not\in G]$. Then $T = R[Y]$ is a polynomial ring over $R$ and $S[x_i^{-1}|\ i \in G] = T[x_i^{\pm 1}|\ x_i \in G]$ is a Laurent polynomial ring over $T$.
Since $I_{\link_\D G}$ is an ideal in $R$, the variables of $Y$ forms a regular sequence on $I_{\link_\D G}T$. Therefore, properties of $I_\D S[x_i^{-1}|\ i \in G]$ and $I_{\link_\D G}$ are strongly related to each other. So we may think of $I_{\link_\D G}$ as the combinatorial localization of $\D$ at $G$. }
\end{Remark}

For simplicity we say that a simplicial complex $\D$ is a {\it complete intersection} if $I_\D$ is a complete intersection. Combinatorially, this means that the minimal nonfaces of $\D$ are disjoint.  We say that $\D$ is {\it locally a complete intersection} if $\link_\D\{i\}$ is a complete intersection for $i = 1,...,n$.\smallskip 
 
There are the following relationship between complete intersections and locally complete intersections,
which will play an essential role in our investigation on ordinary powers.

 \begin{Lemma} \label{OrdLocal} \cite[Theorem 1.5]{TY}
 Let $\D$ be a simplicial complex with $\dim \D \ge 2$.
 Then $\D$ is a complete intersection if and only if $\D$ is connected and locally a complete intersection.
 \end{Lemma}
 
 \begin{Lemma}\label{LCI}\cite[Theorem 1.15]{TY}
Let $\Delta$ be a simplicial complex with $\dim \D \ge 2$. 
Then $\D$ is locally a complete intersection if and only if $\D$ is a union of disjoint complete intersections.
\end{Lemma}

Now we want to prove similar results for matroids. 
Recall that a {\em matroid} is a collection of subsets of a finite set, called  independent sets, with the following properties: \smallskip

(i) The empty set is independent. \par
(ii) Every subset of an independent set is independent. \par
(iii) If $F$ and $G$ are two independent sets and $F$  has more elements than $G$, then there exists an element in $F$ which is not in $G$ that when added to $G$ still gives an independent set. 
\smallskip

We may consider a matroid as a simplicial complex.
It is easy to see that induced subcomplexes, stars and links of faces of matroids are again matroids.
Moreover, every matroid is a pure complex, that is, all facets have the same dimension. \smallskip

We shall need the following criterion for a simplicial complex to be a matroid.

\begin{Lemma} \label{matroid} \cite[Theorem 39.1]{Sch}
A simplicial complex $\Delta$ is a matroid if and only if for any pair of faces 
$F, G$  of $\Delta $ with $\vert F\setminus G \vert =1$ 
and $\vert G\setminus F \vert =2$, there is a vertex $i \in G\setminus F $
with $F \cup \{i\}  \in \Delta $.
\end{Lemma}

For a graph this characterization can be reformulated as follows.

\begin{Corollary} \label{matroid graph}
 A graph $\Gamma$ without isolated vertices is a matroid if and only if every pair of disjoint edges is contained in a 4-cycle. 
\end{Corollary}

\begin{proof}
Assume that $\G$ is a matroid. By Theorem \ref{matroid}, every vertex $i$ is connected to every edge $\{u,v\}$  by an edge (of $\G$). Let $\{i,j\}$ and $\{u,v\}$ be two disjoint edges.
We may assume that $\{i,u\}$ is an edge.  If $\{j,v\}$ is an edge, then $i,u,v,j$ are the ordered vertices of a 4-cycle of $G$. If $\{j,v\}$ is not an edge, then $\{j,u\}$ must be an edge. Since either $i$ or $j$ is connected with $v$ by an edge, we conclude that $\{i,j\}$ and $\{u,v\}$ are always contained in a 4-cycle.\par

Conversely, assume that every pair of disjoint edges is contained in a 4-cycle. Let $i$ be an arbitrary vertex and $\{u,v\}$ an edge not containing $i$. If $i$ isn't connected with $\{u,v\}$ by an edge, there is an edge $\{i,j\}$ such that $j \neq u,v$. But then $\{i,j\}$ and $\{u,v\}$ are contained in a 4-cycle so that $i$ is connected with $\{u,v\}$ by an edge, which is a contradiction.
\end{proof}

We say that $\Delta$  is  \textit{locally a matroid} if $\link_{\Delta}\{i\}$ is a matroid for every vertex $i$ of $\D$. This notion will play an essential role in our investigation on symbolic powers. 

\begin{Theorem}\label{SymLocal}
Let $\Delta$ be  a simplicial complex with $\dim \Delta \ge 2$.
Then $\Delta$ is a matroid if and only if $\Delta$ is connected and locally a matroid. 
\end{Theorem}

\begin{proof}
The necessity can be easily seen from the definition of matroids. 
To prove the sufficiency assume that $\Delta$ is connected and locally a matroid. 
\par

We show first that $\D$ is pure. Let $\dim \D = d$.
Assume for the contrary that there is a facet $F$ with $\dim F < d$.
Since $\link_\D\{v\}$ is a matroid for every vertex $v$ and since every matroid is pure, no vertex $v$ of $F$ is contained in a facet $G$ with $\dim G = d$.
Let $\Gamma$ be the graph of  the one-dimensional faces of $\D$. Let $r$ be the minimal length of a path of $\Gamma$ which connects a vertex of $F$ with a vertex of a facet $G$ with $\dim G = d$. Then $r \ge 1$. Let $v_0,...,v_r$ be the ordered vertices of such a path. Let $H$ be a facet containing the edge $\{v_{r-1},v_r\}$. Since $G \setminus \{v_r\}$ and $H \setminus \{v_r\}$ are facets of  the matroid $\link_\D\{x_r\}$, they have the same dimension. Therefore,
$\dim H = \dim G = d$. Since $x_{r-1}$ is a vertex of $H$, we obtain a contradiction to the minimality of $r$.\par

Next we show that two arbitrary vertices $u,v$ are connected by a path of $\Gamma$ of length at most 2.
For that we only need to prove that if $u, v, w,t$ are the ordered vertices of a path of length 3,
then $u$ and $t$ are connected by a path of length 2.
We may assume that $\{u, w\}, \{v, t\} \not\in \Delta$. Since $\Delta $ is pure and   $\dim \Delta \ge 2$,
there is another vertex $s$ such that $\{v,w, s\} \in \Delta$.
Therefore, $\{w, s\} \in \link_{\Delta}\{v\}$.
Since $\link_{\Delta}\{v\}$ is a matroid, so is $\star_{\Delta}\{v\}$.
Therefore,  the graph of the one-dimensional faces of $\star_{\Delta}\{v\}$ is also a matroid. 
By Corollary \ref{matroid graph}, every pair of disjoint edges of $\star_{\Delta}\{v\}$ is contained in a 4-cycle.  Since $\{u, w\} \not\in \Delta$, we must have $\{u, s\} \in \Delta$.
Similarly, we also have  $\{s,t\} \in \Delta$. Hence
$\{u, s\} , \{s, t\} $ form a path of length 2. \par

By Theorem \ref{matroid}, to show that $\Delta$ is a matroid we only need to show that if 
$F$ and $G$ are two faces of $\Delta $ with $\vert F\setminus G \vert =1$ 
and $\vert G\setminus F \vert =2$, then there is a vertex $i \in G\setminus F $
with $F \cup \{i\}  \in \Delta $.
Since $\Delta $ is locally a matroid,
we may assume that $F$ and $G$ are not faces of the link of any vertex. From this it follows that $F \cap G = \emptyset$. Hence $\vert F \vert =1$ and $\vert G \vert =2$. \par

Let $F=\{u\}$ and $G=\{v,w\}$. 
Assume for the contrary that $ \{u, v\}, \{u,w\} \not\in \Delta $.
Since $u,v$ are connected by a path of length 2, 
there is a vertex $t$ such that $ \{u, t\}, \{v, t\} \in \Delta $. 
Since $\{u\} \in \link_\D\{t\}$, $\{v,w\} \not\in \link_\D\{t\}$ by our assumption on $F$ and $G$. This implies $\{t,w\} \not\in \link_\D\{v\}$.
Let $s$ be a vertex of a facet of $\D$ containing $\{t,v\}$.
Since $\{w\}, \{s,t\} \in \link_\D\{v\}$ and since $\link_{\Delta}\{v\} $ is 
a matroid, we must have $\{s,w\} \in \link_{\Delta}\{v\}$ by Theorem \ref{matroid}. This implies $\{v,w\} \in \link_\D\{s\}$. Hence $\{u\} \not\in  \link_\D\{s\}$ by the assumption on $F$ and $G$. On the other hand, since $ \{u\},  \{s,v\} \in \link_{\Delta}\{t\} $ and since $\link_{\Delta}\{t\} $ is a matroid, we must have  $\{s,u\} \in \link_{\Delta}\{t\}$, which implies $\{u\} \in \link_\D\{s\}$,  a contradiction.
\end{proof}

\begin{Corollary} \label{local matroid}
Let $\Delta$ be a pure simplicial complex with $\dim \Delta \ge 2$. Then $\Delta$ is locally a matroid if and only if $\Delta$ is a union of disjoint matroids.
\end{Corollary}

\begin{proof}
Assume that $\D$ is locally a matroid. Let $\G$ be a connected component of $\D$. Since $\D$ is pure, $\dim \G = \dim \D \ge 2$. For any vertex $v$ of $\G$ we have $\link_\G v = \link_\D v$. Hence $\G$ is also locally a matroid. By Theorem \ref{SymLocal}, this implies that $\G$ is a matroid.\par
Conversely, let $\Delta$ be a union of disjoint matroids. For any vertex $v$ of $\D$ we have $\link_\D v = \link_\G v$, where $\G$ is the connected component containing $v$. Since $\G$ is a matroid, $\link_\G v$ is a matroid. Hence $\D$ is locally a matroid.
\end{proof}

Theorem \ref{SymLocal} and Corollary \ref{local matroid} don't hold if $\dim \D = 1$. In fact, any graph is locally a matroid but there are plenty connected graphs which aren't matroids.

\section{Cohen-Macaulayness of large symbolic powers}

Let $\D$ be a simplicial complex on the vertex set $V(\D) = [n]$.
For every number $m \ge 1$, the $m$-th symbolic power of $I_{\Delta}$ is  the ideal
\[
 I_{\Delta}^{(m)} = \bigcap_{F \in \mathcal{F}(\Delta)} P_{\bar F}^{m}.   
\]

By \cite[Theorem 3.5]{MiT2} or \cite[Theorem]{Va},  $I_{\Delta}^{(m)}$ is Cohen-Macaulay for every $m \ge 1$ if and only if $\D$ is a matroid. We will show in this section that $\D$ is a matroid if $I_{\Delta}^{(m)}$ satisfies some weaker property than the Cohen-Macaulayness.
The idea comes from Theorem \ref{SymLocal} which says that a complex is a matroid  if and only if it is locally a matroid and connected. The property of being locally a matroid can be passed to the one-dimensional case where everything is known \cite{MiT1}. Therefore, it remains to study  the connectedness of $\D$. For that we need the following result of Takayama on local cohomology modules of monomial ideals. \smallskip

Let $I$ be a monomial ideal of $S = K[x_1,...,x_n]$. Since $S/I$ has a natural $\NN^n$-graded structure, its local cohomology modules $H_{\frk m}^i(S/I)$ with respect to the ideal ${\frk m}=(x_1,\ldots,x_n)$ have an $\ZZ^n$-graded structure. For every $\a \in \ZZ^n$ let $[H_{\frk m}^i(S/I)]_{{\bf a}}$ denote the $\a$-component of $H_{\frk m}^i(S/I)$.

\begin{Lemma} \label{Takayama} \cite[Theorem 2.2]{Ta}  
There is a simplicial complex $\D_\a$ such that  
$$[H_{\frk m}^i(S/I)]_{{\bf a}} \cong \widetilde{H}^{i-1}(\Delta_{\bf a}, K),$$
where $\widetilde{H}^{i-1}(\Delta_{\bf a}, K)$ is the $(i-1)$th reduced cohomology group of $\D_a$ with coefficients in $K$.
\end{Lemma}

By \cite[Lemma 1.2]{MiT2} the complex $\D_\a$ can be described as follows. For $\a = (a_1,...,a_n)$ let $G_\a = \{i|\ a_i < 0\}$ and $x^\a = x_1^{a_1}\cdots x_n^{a_n}$.
Then $\D_\a$ is the complex of all sets of the form $F \setminus G_\a$, where $F$ is a subset of $[n]$ containing $G_\a$ such that $x^\a \not\in IS[x_i^{-1}|\ i \in F]$.
\smallskip

\begin{Example} \label{D0}
{\rm For $\a = \0$ we have $G_\0 = \emptyset$ and $x_\0 = 1$. Hence
$\D_\0$ is the complex of all $F \subseteq [n]$ such that $1 \not\in IS[x_i^{-1}|\ i \in F]$ or, equivalently, $1 \not\in \sqrt{I}S[x_i^{-1}|\ i \in F]$. Since $\sqrt{I}$ is a squarefree monomial ideal, we can find a complex $\D$
such that 
$$\sqrt{I} = I_\D = \bigcap_{G \in \F(\D)}P_{\overline G}.$$
Thus, $F \in \D_0$ if and only if $F \cap \overline G = \emptyset$ for all $G \in \F(\D)$. Hence $\D_\0 = \D$. }
\end{Example}

\begin{Lemma}\label{S2}
Let $\Delta$ be  a simplicial complex with $\dim \Delta \ge 1$ and $I$ a monomial ideal with $\sqrt{I} = I_\D$.
Then $\Delta$ is connected if $\depth S/I \ge 2$. 
\end{Lemma}

\begin{proof}
If $\depth S/I \ge 2$, we have $H_\mm^1(S/I) = 0$. By Lemma \ref{Takayama}, this implies $\widetilde H^0(\D_\0,K) = 0$. By the above example we know that $\D_\0 = \D$. Hence $\D$ is connected.
\end{proof}

We say that an ideal $I$ in $S$ (or $S/I$) satisfies {\it Serre condition} $(S_2)$ if $\depth (S/I)_P \ge \min\{2,\height P\}$ for every prime ideal $P$ of $S$. \smallskip

The condition $(S_2)$ as well as the Cohen-Macaulayness of $I_\D^{(m)}$ can be passed to the links of $\D$. To see that we need the following observation.

\begin{Lemma}\label{SymExt}
Let $I$ be a squarefree monomial ideal in a polynomial ring $R$.
Let $T:=R[y]$ be a polynomial ring over $R$. Then\par
{\rm (i)} $(I,y)^{(m)}$ satisfies $(S_2)$ if and only if $I^{(k)}$ satisfies $(S_2)$ for every $k$
with $1 \le k \le m$.\par
{\rm (ii)} $(I,y)^{(m)}$ is Cohen-Macaulay if and only if $I^{(k)}$ is Cohen-Macaulay for every $k$
with $1 \le k \le m$.
\end{Lemma}

\begin{proof} 
Let $I=\cap_j P_j$ be the minimal prime decomposition of $I$. 
Since $(I,y)T=\cap_j (P_j,y)$ is the minimal prime decomposition of $(I,y)T$,
\[
(I,y)^{(m)}  =  \bigcap_j (P_j,y)^{m} 
 =  \bigcap_j \sum_{k=0}^{\ell} P_j^{k} y^{m-k} 
 =  \sum_{k=0}^{m} \left(\cap_i P_j^{k}\right) y^{\ell-k} 
 =  \sum_{k=0}^{m} I^{(k)} y^{m-k}.  
\]
Therefore, 
\[
T/(I,y)^{(m)}   \cong R/I \oplus R/I^{(2)} \oplus \cdots \oplus R/I^{(m)}
\]
as $R$-modules. The assertions (i) and (ii) follow from this isomorphism.
\end{proof}

\begin{Corollary}\label{SymLink}
Let $G$ be a face of $\D$. Then\par
{\rm (i)}  $I_\D^{(m)}S[x_i^{-1}|\ i \in G]$ satisfies $(S_2)$ if and only if $I_{\link_\D G}^{(k)}$
satisfies $(S_2)$ for every $k$ with $1 \le k \le m$.\par
{\rm (ii)} $I_\D^{(m)}S[x_i^{-1}|\ i \in G]$ is Cohen-Macaulay if and only if $I_{\link_\D G}^{(k)}$ is Cohen-Macaulay for every $k$ with $1 \le k \le m$.
\end{Corollary}

\begin{proof} 
Let $Y  = \{x_i|\ i \not \in V(\star_{\Delta}G)\}$ and $T = K[x_i|\ x_i \not\in G]$.
By Lemma \ref{local} and Remark \ref{link}, $I_\D^{(m)}S[x_i^{-1}|\ i \in G]$ satisfies $(S_2)$ if and only if
$(I_{\link_\D G},Y)^{(m)}T$ satisfies $(S_2)$.
By Lemma \ref{SymExt}, $(I_{\link_\D G},Y)^{(m)}T$ satisfies $(S_2)$ if and only if $I_{\link_\D G}^{(k)}$   satisfies $(S_2)$ for every $k$ with $1 \le k \le m$. This proves (i). The proof of (ii) is similar.
\end{proof} 
 
Now we are able to prove the following characterization for the Cohen-Macaulayness of a symbolic power $I_\D^{(m)}$, $m \ge 3$.  

\begin{Theorem}\label{SymCM}
Let $\Delta$ be a simplicial complex with $\dim \D \ge 1$.
Then the following conditions are equivalent: \par
{\rm (i)} $I_{\Delta }^{(m)}$ is Cohen-Macaulay for every $m \ge 1$.\par
{\rm (ii)} $I_{\Delta }^{(m)}$ is Cohen-Macaulay for some $m \ge 3$.\par
{\rm (iii)} $I_{\Delta }^{(m)}$ satisfies $(S_2)$ for some $m \ge 3$.\par
{\rm (iv)} $\Delta$ is a matroid.
\end{Theorem}

\begin{proof}
(i)$\Rightarrow$(ii)$\Rightarrow$(iii) is clear. 

(iii)$\Rightarrow$(iv). 
If $\dim \D = 1$,  $(S_2)$ means that  $I_{\Delta }^{(m)}$ is Cohen-Macaulay.
In this case, the assertion follows from \cite[Theorem 2.4]{MiT1}. 
Let $\dim \D \ge 2$. By Theorem \ref{SymLocal} we only need to show that $\D$ is connected and locally a matroid. Since $(S_2)$ implies $\depth S/I_\D^{(m)} \ge 2$, $\D$ is connected by Lemma \ref{S2}. By Corollary \ref{SymLink}, $I_{\link_\D\{i\} }^{(m)}$ satisfies $(S_2)$ for all $i = 1,...,n$. Using induction on $\dim \D$ we may assume that $\link_\D\{i\}$ is a matroid for all $i = 1,...,n$. Hence $\D$ is locally a matroid.

(iv)$\Rightarrow$(i) follows from \cite[Theorem 3.5]{MiT2} or \cite[Theorem 2.1]{Va}.
\end{proof}

The implication (ii) $\Rightarrow$ (i) gives a positive answer to the question of  \cite{MiT2} whether there exists a number $t$ depending on $\dim \Delta$ such that if $I_{\Delta }^{(t)}$ is Cohen-Macaulay, then  
$I_{\Delta }^{(m)}$ is Cohen-Macaulay for every $m \ge 1$.
As the Cohen-Macaulayness of $I_\D^{(2)}$ implies that of $I_{\Delta }$ \cite[Theorem 3.7]{HHT} we also obtain a positive answer to the question of \cite{MiT2} whether the Cohen-Macaulayness of $I_{\Delta }^{(m+1)}$ implies that of  $I_{\Delta }^{(m)}$ for every $m \ge 1$. \smallskip

Combinatorial criteria for the Cohen-Macaulayness of $I_\D^{(2)}$ can be found in \cite{MiT1} for $\dim \D= 1$ and in \cite{MiT2} for arbitrary dimension. Using these criteria one can easily find simplicial complexes $\D$ such that $I_\D^{(2)}$ is Cohen-Macaulay but $I_\D^{(m)}$ is not Cohen-Macaulay for every $m\ge 3$. Such an example is the 5-cycle, which is not a matroid.
\smallskip

Next we consider the generalized Cohen-Macaulayness of large symbolic powers of Stanley-Reisner ideals.
\smallskip

Recall that a homogeneous ideal  $I$ in $S$ (or $S/I$) is called \textit{generalized Cohen-Macaulay} if
$\dim_K  H_{\frak m}^i(S/I) < \infty $ for $i=0,1,\ldots,\dim S/I-1$. 
It is well-known that $I$ is generalized Cohen-Macaulay if and only if 
$IS[x_i^{-1}]$ is Cohen-Macaulay for $i = 1,...,n$ and $S/I$ is equidimensional, that is, $\dim S/P = \dim S/I$ for every minimal prime $P$ of $I$ (see \cite{CST}). The class of generalized Cohen-Macaulay ideals is rather large. For instance, $I_{\Delta }^{(m)}$ is generalized Cohen-Macaulay for every $m \ge 1$ if $\D$ is pure and $\dim \D = 1$. For $\dim \D \ge 2$, the situation is completely different.

\begin{Theorem}\label{SymFLC}
Let $\Delta$ be a simplicial complex with $\dim \Delta \ge 2.$
Then the following conditions are equivalent:\par
{\rm (i)} $I_{\Delta }^{(m)}$ is generalized Cohen-Macaulay for every $m \ge 1$.\par
{\rm (ii)} $I_{\Delta }^{(m)}$ is generalized Cohen-Macaulay for some $m \ge 3$.\par
{\rm (iii)} $\Delta$ is a union of disjoint matroids of the same dimension.
\end{Theorem}

\begin{proof}
(i)$\Rightarrow$(ii) is clear.

(ii)$\Rightarrow$(iii).
Since $I_{\Delta }^{(m)}S[x_i^{-1}]$ is Cohen-Macaulay for $i = 1,...,n$, 
$I_{\link_{\Delta}\{i\} }^{(m)}$ is Cohen-Macaulay by Corollary \ref{SymLink}. 
By Theorem \ref{SymCM}, $\link_{\Delta}\{i\}$ is a matroid for $i = 1,...,n$.
Hence $\Delta$ is locally a matroid.
By Corollary \ref{local matroid},  $\Delta$ is a union of disjoint matroids. 
On the other hand, since $S/I_\D$ is equidimensional, $\D$ is pure. 
Therefore, the connected components of $\D$ have the same dimension.

(iii)$\Rightarrow$(i). 
By Corollary \ref{local matroid}, $\D$ is locally a matroid.
Since $\link_{\Delta}\{i\} $ is a matroid for $i = 1,...,n$,  
$I_{\link_{\Delta}\{i\} }^{(m)}$ is Cohen-Macaulay for every $m \ge 1$ by Theorem \ref{SymCM}.
By Corollary  \ref{SymLink}, this implies the Cohen-Macaulayness of  $I_{\Delta }^{(m)}S[x_i^{-1}]$ for for every $m \ge 1$, $i = 1,...,n$. Since $\D$ is pure, $S/I_\D$ is equidimensional.
Therefore, $I_{\Delta }^{(m)}$ is generalized Cohen-Macaulay for every $m \ge 1$.
\end{proof}

A homogeneous ideal $I$ in $S$ (or $S/I$) is called {\it Buchsbaum} or {\it quasi-Buchsbaum} if the natural map $\text{Ext}_S^i(K,S/I) \to H_\mm^i(S/I)$  is surjective or $\mm H_\mm^i
(S/I) = 0$ for $i = 0,...,\dim S/I-1$ (see e.g. \cite{SV} and \cite{Go2}). We have the following implications:
\smallskip

\centerline{Cohen-Macaulayness $\Rightarrow$ Buchsbaumness $\Rightarrow$ quasi-Buchsbaumness}\par
\centerline{$\Rightarrow$ generalized Cohen-Macaulayness.}
\medskip

We will use Theorem \ref{SymFLC} to study the Buchsbaumness and quasi-Buchsbaumness of large symbolic powers of Stanley-Reisner ideals. For that we need the following observation.

\begin{Lemma} \label{SymNonQBbm}
Let $\Delta$ be a simplicial complex with $\dim \Delta \ge 1.$ Then $\D$ is connected if 
$I_{\Delta }^{(m)}$ is quasi-Buchsbaum for some $m \ge 2$.
\end{Lemma}

\begin{proof}
Set $I=I_{\Delta }^{(m)}$ and $\e = (1,0,...,0)$. We have $\D_0 = \D$ by Example \ref{D0} and
$$\Delta_{\bf e} =\{F \subset [n]|\ x_1 \not\in IS[x_i|\ i \in F]\}.$$
For $F \in \F(\D)$ we have $IS[x_i|\ i \in F] = P_{\overline F}^mS[x_i|\ i \in F]$. 
Since  $x_1 \not\in P_{\overline F}^mS[x_i|\ i \in F]$ ($m \ge 2$) we get $F \in \D_\e$. For $F \not\in \D$ we have $IS[x_i|\ i \in F] = S[x_i|\ i \in F]$. Hence $\D_{\e} = \D$.
By \cite[Lemma 2.3]{MN1} there is a commutative diagram

\begin{picture}(400,70)
\put(120,50){$H_{\frk m}^1(S/I)_{\bf 0}$} 
\put(185,55){\vector(1,0){40}} 
\put(145,42){\vector(0,-1){20}}
\put(200,58){$x_1$}
\put(240,50){$H_{\frk m}^1(S/I)_{\bf e}$} 
\put(110,10){$\widetilde{H}^0(\Delta_{\bf 0},K)$} 
\put(190,15){\vector(1,0){35}} 
\put(270,42){\vector(0,-1){20}}
\put(240,10){$\widetilde{H}^0(\Delta_{\bf e},K)$,} 
\end{picture}
\par \vspace{1mm}

\noindent where the vertical maps are the isomorphisms of Lemma \ref{Takayama}
and the lower horizontal map is induced from the natural embedding 
$\Delta_{{\bf e}} \hookrightarrow  \Delta_\0$.
Since $\D_\e = \D_\0$, this map is an identity. Therefore,
$$
 \widetilde{H}_0(\Delta ,K)  \cong H_{\frk m}^1(S/I)_{\bf e} = x_1H_{\frk m}^1(S/I)_{\bf 0} = 0
$$
because $I$ is quasi-Buchsbaum. The vanishing of  $\widetilde{H}_0(\Delta ,K)$ just means that $\D$ is connected.
\end{proof}

The above lemma doesn't hold for $m = 1$. 
It is well known that if $I_\D$ is generalized Cohen-Macaulay, then $I_\D$ is Buchsbaum \cite[Theorem 3.2]{Sc}, \cite[Theorem 8.1]{St}.
Hence $I_\D$ is Buchsbaum if $\dim \D = 1$, even when $\D$ is unconnected. \smallskip

The following result shows that the Cohen-Macaulayness of $I_{\Delta }^{(m)}$, $m \ge3$, is equivalent to the Buchsbaumness and quasi-Buchsbaumness.

\begin{Theorem}\label{SymBbm}
Let $\Delta$ be a pure simplicial complex with $\dim \Delta \ge 2.$
Then the following conditions are equivalent:\par
{\rm (i)} $I_{\Delta }^{(m)}$ is Cohen-Macaulay for every $m \ge 1$.\par
{\rm (ii)} $I_{\Delta }^{(m)}$ is Buchsbaum for some $m \ge 3$.\par
{\rm (iii)} $I_{\Delta }^{(m)}$ is quasi-Buchsbaum for some $m \ge 3$.\par
{\rm (iv)} $\Delta$ is a matroid.
\end{Theorem}

\begin{proof}
(i)$\Rightarrow$(ii)$\Rightarrow$(iii) is clear.

(iii)$\Rightarrow$(iv).
The quasi-Buchsbaumness implies that $I_\Delta^{(m)}$  is generalized Cohen-Macaulay.
By Theorem \ref{SymFLC}, $\Delta$ is a union of disjoint matroids. 
On the other hand,  $\D$ is connected by Lemma \ref{SymNonQBbm}. Hence  $\D$ is a matroid.

(iv)$\Rightarrow$(i) follows from \cite[Theorem 3.5]{MiT2} or \cite[Theorem 2.1]{Va}.
\end{proof}

The situation is a bit different if $\dim \Delta =1$. 
Minh-Nakamura \cite[Theorem 3.7]{MN1} showed that for a graph $\D$, $I_\D^{(3)}$ is Buchsbaum if and only if $I_\D^{(2)}$ is Cohen-Macaulay and that for $m \ge 4$, $I_\D^{(m)}$  is Cohen-Macaulay if $I_\D^{(m)}$ is Buchsbaum.
For instance, if $\D$ is a 5-cycle, $I_{\Delta }^{(3)}$ is Buchsbaum but not Cohen-Macaulay.

\section{Cohen-Macaulayness of large ordinary powers}

In this section we study the Cohen-Macaulayness of ordinary powers of Stanley-Reisner ideals. \smallskip

 It is well known that $I_\D^m$ is Cohen-Macaulay for every $m \ge 1$ if and only if $\D$ is a complete intersection \cite{CN}. If $\dim \D \le 2$, we know that $\D$ is a complete intersection if $I_\D^m$ is Cohen-Macaulay for some $m \ge 3$ \cite{MiT1}, \cite{TrTu}. It was asked whether this result holds in general \cite[Question 3]{TrTu}. We shall give a positive answer to this question by showing that $\D$ is a complete intersection if $S/I_\D^m$ satisfies Serre condition $(S_2)$ for some $m \ge 3$.
 \smallskip
 
To study the relationship between the ordinary powers of $\D$ and of its links we need the following observation. 

\begin{Lemma}\label{OrdExt} 
Let $I$ be a monomial ideal in a polynomial ring $R$. Let $T:=R[y]$ be a polynomial ring over $R$. Then\par
{\rm (i)} $(I,y)^{m}$ satisfies $(S_2)$ if and only if $I^k$ satisfies $(S_2)$ for every $k$ with $1 \le k \le m$.\par
{\rm (ii)} $(I,y)^{m}$ is Cohen-Macaulay if and only if $I^k$ is Cohen-Macaulay for every $k$ with $1 \le k \le m$.
\end{Lemma}

\begin{proof} 
Since 
$
(I,y)^{m} 
 =  \sum_{k=0}^{m} I^{k} y^{m-k} 
$, we have
\[
T/(I,y)^{m}   \cong R/I^{m} \oplus R/I^{m-1} \oplus \cdots \oplus R/I
\]
as $R$-modules. 
The assertion follows from this isomorphism.
\end{proof}

\begin{Corollary}\label{OrdLink}
Let $G$ be a face of $\D$. Then\par
{\rm (i)}  $I_\D^{m}S[x_i^{-1}|\ i \in G]$ satisfies $(S_2)$ if and only if $I_{\link_\D G}^{k}$
satisfies $(S_2)$ for every $k$ with $1 \le k \le m$.\par
{\rm (ii)} $I_\D^{m}S[x_i^{-1}|\ i \in G]$ is Cohen-Macaulay if and only if $I_{\link_\D G}^{k}$ is Cohen-Macaulay for every $k$ with $1 \le k \le m$.
\end{Corollary}

\begin{proof} 
The assertions follow from Lemma \ref{OrdExt} similarly as in the proof of 
Corollary \ref{SymLink}.
\end{proof} 

Now we are able to prove the following characterizations of the Cohen-Macaulayness of $I_\D^m$, $m \ge 3$.

\begin{Theorem}\label{OrdCM}
Let $\Delta$ be a simplicial complex with $\dim \D \ge 1$.
Then the following conditions are equivalent: \par
{\rm (i)} $I_{\Delta }^{m}$ is Cohen-Macaulay for every $m \ge 1$. \par
{\rm (ii)} $I_{\Delta }^{m}$ is Cohen-Macaulay for some $m \ge 3$. \par
{\rm (iii)} $I_{\Delta }^{m}$ satisfies $(S_2)$ for some $m \ge 3$. \par
{\rm (iv)} $\Delta$ is a complete intersection.
\end{Theorem}

\begin{proof}
(i)$\Rightarrow$(ii)$\Rightarrow$(iii) and (iv)$\Rightarrow$(i) are clear. \par

(iii)$\Rightarrow$(iv). 
If $\dim \D = 1$,  $(S_2)$ means that  $I_{\Delta }^{m}$ is Cohen-Macaulay.
In this case, the assertion follows from \cite[Corollary 3.5]{MiT1}. 
Let $\dim \D \ge 2$. By Lemma \ref{OrdLocal} we only need to show that $\D$ is connected and locally a complete intersection. Since $(S_2)$ implies $\depth S/I_\D^{m} \ge 2$, $\D$ is connected by Lemma \ref{S2}. By Corollary \ref{OrdLink}, $I_{\link_\D\{i\} }^{m}$ satisfies $(S_2)$ for all $i = 1,...,n$. Using induction on $\dim \D$ we may assume that $\link_\D\{i\}$ is a complete intersection for all $i = 1,...,n$. Hence $\D$ is locally a complete intersection.
\end{proof}

There is a complete description of all complexes $\D$ such that $I_\D^2$ is Cohen-Macaulay in \cite[Corollary 3.4]{MiT1} for $\dim \D = 1$ and in \cite[Theorem 3.7]{TrTu} for $\dim \D = 2$. Using these results one can find examples such that $I_\D^2$ is Cohen-Macaulay but $I_\D^m$ is not Cohen-Macaulay for every $m \ge 3$. The 5-cycle is such an example. 
\smallskip

Next we consider the generalized Cohen-Macaulayness of $I_\D^m$. Goto-Takayama \cite[Theorem 2.5]{GT} showed that $I_{\Delta }^{m}$ is generalized Cohen-Macaulay for every integer $m \ge 1$ if and only if $\Delta$ is pure  and locally a complete intersection. 
This result can be improved in the case $\dim \D = 1$ as follows.

\begin{Theorem}
Let $\Delta$ be a graph.
Then the following conditions are equivalent:\par
{\rm (i)} $I_{\Delta }^{m}$ is generalized Cohen-Macaulay for every $m \ge 1$.\par
{\rm (ii)} $I_{\Delta }^{m}$ is generalized Cohen-Macaulay for some $m \ge 3$.\par
{\rm (iii)} $\Delta$ is a union of disjoint paths and cycles.
\end{Theorem}

\begin{proof}
(i)$\Rightarrow$(ii) is clear. \par

(ii)$\Rightarrow$(iii). 
The generalized Cohen-Macaulayness of $I_{\Delta }^{m}$ implies that  $I_\D^mS[x_i^{-1}]$ is Cohen-Macaulay for $i = 1,...,n$. Therefore, $I_{\link_{\Delta}\{i\} }^{m}$ is Cohen-Macaulay by Corollary \ref{OrdLink}. Since $\dim \D = 1$, $\link_{\Delta}\{i\}$ is a collection of vertices. Hence
$I_{\link_{\Delta}\{i\} }$ is the edge ideal of a complete graph. By  \cite[Theorem 3.8]{RTY1}, the Cohen-Macaulayness of $I_{\link_{\Delta}\{i\} }^{m}$ for some $m \ge 3$ implies that 
$I_{\link_{\Delta}\{i\}}$ is a complete intersection. Thus, $\link_{\Delta}\{i\} $ consists of either one point or two points. From this it follows that every connected component of $\D$ must be a path or a cycle.  \par

(iii)$\Rightarrow$(i).  Condition (iii) implies that $\link_{\Delta}\{i\}$ consists of either one point or two points for $i = 1,...,n$. Hence  $\D$ is locally a complete intersection. Therefore, $I_\D^m$ is generalized Cohen-Macaulay for every $m \ge 1$ by \cite[Theorem 2.5]{GT}. 
\end{proof}

For $\dim \D \ge 2$ we have the following characterization.

\begin{Theorem}\label{OrdFLC}
Let $\Delta$ be a simplicial complex with $\dim \D \ge 2$.
Then the following conditions are equivalent:\par
{\rm (i)} $I_{\Delta }^{m}$ is generalized Cohen-Macaulay for every $m \ge 1$.\par
{\rm (ii)} $I_{\Delta }^{m}$ is generalized Cohen-Macaulay for some $m \ge 3$.\par
{\rm (iii)} $\Delta$ is a union of disjoint complete intersections of the same dimension.
\end{Theorem}

\begin{proof}
(i)$\Rightarrow$(ii) is clear. \par

(ii)$\Rightarrow$(iii). 
The generalized Cohen-Macaulayness of $I_{\Delta }^{m}$ implies that  $I_\D^mS[x_i^{-1}]$ is Cohen-Macaulay for $i = 1,...,n$ and $I_\D$ is equidimensional. The last property means that $\D$ is pure.
By Corollary \ref{OrdLink}, the Cohen-Macaulayness of $I_\D^mS[x_i^{-1}]$ implies that of $I_{\link_{\Delta}\{i\} }^{m}$. Hence $\link_{\Delta}\{i\} $ is a complete intersection by Theorem \ref{OrdCM}. Thus, $\D$ is locally a complete intersection. By Lemma \ref{LCI}, $\Delta$ is a union of disjoint complete intersections. Since $\D$ is pure, these complete intersections have the same dimension. \par

(iii)$\Rightarrow$(i). By Lemma \ref{LCI}, $\Delta$ is locally a complete intersection. Hence $I_{\link_\D\{x_i\}}^m$ is Cohen-Macaulay for every $m \ge 1$, $i = 1,...,n$. By Corollary \ref{OrdLink}, this implies the Cohen-Macaulayness of $I_\D^mS[x_i^{-1}]$ for every $m \ge 1$. Since $\D$ is pure, $S/I_\D$ is equidimensional. Therefore, $I_\D^m$ is generalized Cohen-Macaulay for every $m \ge 1$. 
\end{proof}

The above two theorems are similar but they can't be put together because a path of length $\ge 4$ or a cycle of length $\ge 5$ is not a complete intersection. Compared with the mentioned result of Goto-Takayama, these theorems are stronger in the sense that it gives a combinatorial characterization of the generalized Cohen-Macaulayness of each power $I_{\Delta }^{m}$, $m \ge 3$.  
\smallskip

Now we consider the Buchsbaumness and quasi-Buchsbaumness of ordinary powers of Stanley-Reisner ideals.
By \cite[Theorem 2.1]{TY} we know that $\Delta$ is a complete intersection if $I_{\Delta }^{m}$ is Buchsbaum for all (large) $m \ge 1$. It was asked \cite[Question 2.10]{TY} whether $\Delta$ is a complete intersection if $I_{\Delta }^{m}$ is quasi-Buchsbaum for every $m \ge 1$. In the following we give a positive answer to this question. \smallskip

The case $\dim \D = 1$ was already studied by Minh-Nakamura in \cite{MN2}, where they describe all graphs $\D$ with a Buchsbaum ideal $I_\D^m$. In particular, they showed that $I_{\D}^{m}$ is Buchsbaum for some $m \ge 4$ if and only if $\D$ is a complete intersection. We extend this result for the quasi-Buchsbaumness as follows.

\begin{Theorem}\label{OrdBbm 1}
Let $\Delta$ be a graph. Then the following conditions are equivalent:\par
{\rm (i)} $I_{\Delta }^{m}$ is Cohen-Macaulay for every $m \ge 1$.\par
{\rm (ii)} $I_{\Delta }^{m}$ is Buchsbaum for some $m \ge 4$.\par
{\rm (iii)} $I_{\Delta }^{m}$ is quasi-Buchsbaum for some $m \ge 4$.\par
{\rm (iv)} $\Delta$ is a complete intersection.
\end{Theorem}

\begin{proof}
We only need to prove (iii)$\Rightarrow$(iv). Consider the exact sequence
$$0 \to I_{\Delta }^{(m)}/I_\D^m \to S/I_\D^m \to S/I_{\Delta }^{(m)} \to 0.$$
Since $H_\mm^0(S/I_\D^m) = I_{\Delta }^{(m)}/I_\D^m$, $H_\mm^0(S/I_\D^{(m)}) = 0$ and $H_\mm^i(S/I_\D^{(m)}) = H_\mm^i(S/I_\D^m)$ for $i \ge 1$. Hence
the quasi-Buchsbaumness of $I_{\Delta }^{m}$  implies that $\mm H_\mm^1(S/I_{\Delta }^{(m)}) = 0$.
Since $\dim S/I_{\Delta }^{(m)} = 2$, $I_{\Delta }^{(m)}$ is Buchsbaum by \cite[Proposition 2.12]{SV}.
By \cite[Theorem 3.7]{MN1}, this implies the Cohen-Macaulayness of $I_{\Delta }^{(m)}$.
Hence every pair of disjoint edges of $\D$ is contained in a 4-cycle by \cite[Theorem 2.4]{MiT1}. On the other hand, $\D$ is locally a complete intersection by Goto-Takayama \cite[Theorem 2.5]{GT}. By \cite[Proposition 1,11]{TY}, $\D$ is either an $r$-cycle, $r \ge 3$, or a path of length $r \ge 2$. Thus, $\D$ must be an $r$-cycle, $r \le 4$, or a path of length 2, which are complete intersections.
\end{proof} 

We can not lower $m$ to 3 in (ii) and (iii) of the above theorem. In fact, if $\D$ is a 5-cycle, then $I_\D^3$ is Buchsbaum but not Cohen-Macaulay by \cite[Theorem 4.11]{MN2}.
\smallskip 

For  $\dim \D \ge 2$ we need the following observation.

\begin{Lemma} \label{OrdQB}
Let $\Delta$ be a  simplicial complex with $\dim \Delta \ge 1.$ Then $\D$ is connected if
$I_{\Delta }^{m}$ is quasi-Buchsbaum for some $m \ge 2$.
\end{Lemma}

\begin{proof}
This can be shown similarly as for Lemma \ref{SymNonQBbm}.
\end{proof}

\begin{Theorem}\label{OrdBbm}
Let $\Delta$ be a simplicial complex with $\dim \Delta \ge 2.$
Then the following conditions are equivalent:\par
{\rm (i)} $I_{\Delta }^{m}$ is Cohen-Macaulay for every $m \ge 1$.\par
{\rm (ii)} $I_{\Delta }^{m}$ is Buchsbaum for some $m \ge 3$.\par
{\rm (iii)} $I_{\Delta }^{m}$ is quasi-Buchsbaum for some $m \ge 3$.\par
{\rm (iv)} $\Delta$ is a complete intersection.
\end{Theorem}

\begin{proof}
(i)$\Rightarrow$(ii)$\Rightarrow$(iii) is clear.

(iii)$\Rightarrow$(iv).
Since $S/I_{\Delta }^{m}$ is generalized Cohen-Macaulay, $\D$ is a union of disjoint complete intersections by Theorem \ref{OrdFLC}.
By Lemma \ref{OrdQB}, the quasi-Buchsbaumness of $I_{\Delta }^{m}$ implies that $\Delta$ is connected.
Therefore, $\D$ is a complete intersection.

(iv)$\Rightarrow$(i) is well known. 
\end{proof}

Below we give an example where $\dim \D = 2$ and $I_{\Delta }^2$ is Buchsbaum but not Cohen-Macaulay.

\begin{Example}
{\rm Let $\D$ be the simplicial complex generated by the sets
$$\{1,2,3\},\{1,2,4\},\{1,3,4\},\{2,3,4\},\{3,4,5\}$$ 
and $S = K[x_1,x_2,x_3,x_4,x_5]$. It is easy to check that
\begin{align*}
I_\D^2 & = (x_1x_2x_3x_4,x_1x_5,x_2x_5)^2,\\
I_\D^{(2)} & = (x_1x_2x_3x_4,x_1x_5,x_2x_5)^2 + (x_1x_2x_3x_4x_5).
\end{align*}
From this it follows that $\mm I_\D^{(2)} \subseteq I_\D^2$. On ther hand, $I_\D^{(2)}$ is Cohen-Macaulay by \cite[Example 4.7]{MiT2}. This implies $H_\mm^0(S/I_\D^2) = I_\D^{(2)}/I_\D^2$, hence
$\mm H_\mm^0(S/I_\D^2) = 0$. Therefore, $I_\D^2$ is Buchsbaum \cite[Proposition 2.12]{SV} and not Cohen-Macaulay.}
\end{Example}

\section{Applications}

In this section we investigate the Cohen-Macaulayness of symbolic powers of the cover ideal  and the facet ideal of a simplicial complex.
\smallskip

Let $\D$ be a simplicial complex on the vertex set $V= [n]$ and $S = k[x_1,...,x_n]$. 
The {\it facet ideal} $I(\D)$ is defined as the ideal generated by all squarefree monomials
$x_{i_1}\cdots x_{i_r}$, $\{i_1,...,i_r\} \in \F(D)$ (see \cite{F}).  For instance, squarefree monomial ideals generated in degree $r$ are just facet ideals of pure complexes of dimension $r-1$. 
\smallskip

Let $\D^*$ denote the simplicial complex with $I_{\D^*} = I(\D)$.
By Theorem \ref{SymCM}, to study the Cohen-Macaulayness of large symbolic powers of $I(\D)$ we have to study when $\D^*$ is a matroid. Note that facets of $\D$ are minimal nonfaces of $\D^*$ and that $\D$ and $\D^*$ have the same vertex set. \smallskip

If $\dim \D = 1$, we may consider $\D$ as a graph and and $I(\D)$ as the Stanley-Reisner ideal of a flag complex.
This case was already dealt with in \cite{RTY1}.  

\begin{Theorem}\cite[Theorem 3.6]{RTY1}
Let $\G$ be a graph.
Then the following conditions are equivalent:\par
{\rm (i)} $I(\G)^{(m)}$ is Cohen-Macaulay for every $m\ge 1$.\par
{\rm (ii)} $I(\G)^{(m)}$ is Cohen-Macaulay for some $m\ge 3$.\par
{\rm (iii)} $\G$ is a union of disjoint complete graphs.
\end{Theorem}

The original proof of this theorem didn't involve matroids.
However, using a recent result in matroid theory we can give another proof as follows.

\begin{proof}
It suffices to show that $\G^*$ is a matroid if and only if $\G$ is a union of disjoint complete graphs.
Note first that $\G^*$ is the clique complex of the graph $\overline \G$ of the nonedges of $\G$. 
By \cite[Theorem 3.3]{KOU}, the clique complex of a graph is a matroid if and only if there is a partition of the vertices into independent sets (which contain no adjacent vertices) such that every nonedge of the graph is contained in an independent set. But an independent set of $\overline \G$ is just a complete graph in $\G$. 
\end{proof}

Now we are going to prove  a similar characterization for the case that $\D$ is a pure complex with $\dim \D = 2$.
\smallskip

In the following we call a simplicial complex \textit{$r$-uniform} if it is generated by the $r$-dimensional faces of a simplex. Complete graphs are just $1$-uniform matroids. 

\begin{Theorem} \label{2-uniform}
Let $\D$ be a pure complex with $\dim \D = 2$.
Then the following conditions are equivalent:\par
{\rm (i)} $I(\Delta)^{(m)}$ is Cohen-Macaulay for every $m\ge 1$.\par
{\rm (ii)} $I(\Delta)^{(m)}$ is Cohen-Macaulay for some $m\ge 3$.\par
{\rm (iii)} $\D$ is a union of disjoint 2-uniform matroids.
\end{Theorem}

As observed above, it suffices to show that $\D^*$ is a matroid if and only if $\D$ is a union of disjoint 2-uniform matroids. The proof is based on the following observation.

\begin{Lemma} \label{links}
Let $\Delta$ be a matroid. Let $F, G$ be two maximal proper subsets of a minimal nonface of $\Delta$. 
Then $\link_{\Delta} F = \link_{\Delta} G$.
\end{Lemma}

\begin{proof}
By symmetry, it suffices to show 
that every non-empty face $H$ of $\link_{\Delta}  F$ is also a face of $\link_{\Delta}  G$.
Note that $F \cup H \in \Delta$ and  $F \cap H = \emptyset$. Since $|F \cup H| = |F| + |H| = |G|+|H| > |G|$,
we can extend $G$ by elements of $F \cup H$ to a face $L$ of $\Delta$ such that $|L| = |F \cup H|$.
Since $|F \cap G| = |F| -1$, there is only a vertex of $F$ not contained in $G$.
Since $F \cup G$ is a minimal nonface of $\Delta$, $L$ does not contain this vertex. 
Therefore, $L \subseteq G \cup H$. 
Since $|L| = |G|+|H|$, we must have $L = G \cup H$ and $G \cap H = \emptyset$.
This shows $H \in \link_{\Delta}  G$.
\end{proof}

Using the above lemma we prove  the following structure theorem for matroids whose minimal nonfaces have dimension 2. \smallskip

Recall that the {\it join} $\D_1*\D_2$ of two simplicial complexes $\D_1$ and $\D_2$ on different vertex sets is the simplicial complex whose faces are the unions of two faces of $\D_1$ and $\D_2$. 

\begin{Proposition} \label{decomposition 2}
Let $\D$ be a simplicial complex whose minimal nonfaces have dimension 2.
Then $\D$ is a matroid if and only if $\D$ is the join of 1-uniform matroids with possibly a simplex.
\end{Proposition}

\begin{proof}
It is obvious that the join of two matroids is a matroid. 
Therefore, $\D$ is a matroid if $\D$ is the join of 1-uniform matroids with eventually a simplex. \par
Conversely, assume that $\D$ is a matroid. 
Let $d = \dim \Delta$. Then $d \ge 1$. 
If $d = 1$, $\Delta$ is an $1$-uniform matroid because every set of two vertices of $\D$ is a face of $\D$.
If $d \ge 2$, we choose an edge $F$ which is contained in a minimal nonface $H$ of $\D$. \par
Let $\Delta _1 = \link_\D F$. Then $\dim \D_1 = d-2 \ge 0$. Let $W$ be the vertex set of $\Delta _1$. 
We will show that $\Delta _1$ is the induced subcomplex $\Delta _W$ of $\Delta $ on $W$. 
Assume for the contrary that there is a nonempty face $G$ of $\Delta _W$ such that $F \cup G$ is a nonface of $\Delta $. Since $F \cup \{v\}$ is a face of $\D$ for every vertex $v \in W$, $|G| \ge 2$.
Choose $G$ as small as possible. Then $F \cup G'$ is a face of $\Delta $ for any subset $G' \subset G$ with $|G'| = |G|-1$.  Since $|F \cup G| \ge 4$, $F \cup G$ is not a minimal nonface of $\Delta $. Therefore, there exists a vertex $v \in F$ such that $\{v\} \cup G$ is a nonface of $\Delta $.  By the choice of $G$, this nonface must be minimal. Hence $|\{v\} \cup G| = 2$, which implies $|G| = 2$. Let $u$ be a vertex of $G$. Then $F \cup \{u\}$ is a face of $\Delta $. 
Let $F = \{v,x\}$. By the definition of matroids, $\{x\}\cup G$ must be a face of $\Delta $. By the choice of $F$, there is a minimal nonface $H$ of $\D$ containing $F$. Let $H = \{v,x,y\}$. By Lemma \ref{link}, $\link_\D \{x,y\} = \D_1$. Hence $\{y,u\}$ is a face of $\D$. By the definition of matroids, $\{y\} \cup G$ must be a face of $\D$. 
Since $G \not\in \D_1$, $\{x,y\} \cup G$ is a nonface of $\Delta$. Since every proper subset of $\{x,y\} \cup G$ is a face of $\D$, this nonface is minimal, which is a contradiction to the assumption that every nonface has dimension $2$.
So we have shown that $\D_1$ is an induced subcomplex of $\D$. From this it follows that $\D_1$ is a matroid whose minimal nonfaces have dimension 2. \par

Let $\Gamma _1$ be the induced subcomplex of $\Delta$ on the vertices not contained in $\Delta _1$. 
It is clear that $\Gamma _1$ is a matroid. Since every face of $\D$ containing $F$ properly must contain a vertex of $\D_1$, $F$ is a facet of $\G_1$. Therefore, all facets of $\Gamma _1$ have dimension 1.  
Since every set of $2$ vertices is a face of $\Delta$, $\Gamma _1$ is an 1-uniform matroid. 
We will  show that $\link_\D G = \D_1$ for every facet $G$ of $\Gamma _1$. 
By the definition of matroids, we can find a sequence of facets $F =F_1,F_2,...,F_r = G$ of $\Gamma _1$ 
such that $|F_i \cap F_{i+1}| = 1$, $i =1,...,r-1$. From this it follows that $|F_i \cup F_{i+1}| = 3$.
Since the vertices of $F_2$ are not contained in $\Delta _1$, 
$F_1 \cup F_2$ is a nonface of $\D$. Since $|F_1 \cup F_2| = 3$, $F_1 \cup F_2$ is a minimal nonface of $\D$.  By Lemma \ref{links}, we have $\link_{\Delta } F_1 = \link_{\Delta } F_2 = \D_1$. Similarly, we can show that 
$\link_{\Delta } F_2 = \cdots = \link_{\Delta } F_r = \D_1$. So we obtain $\link_\D G = \D_1$. From this it follows that every facet of $\D$ is a union of a facet of $\G_1$ and a facet of $\D_1$.
 Therefore, $\Delta = \D_1 * \G_1$.\par

If $\Delta _1$ has no nonface, then $\Delta _1$ is a simplex. 
If $\Delta _1$ has a nonface and if $\Delta _1$ is not an 1-uniform matroid,
 we use the above argument to show that there are an 1-uniform matroid $\Gamma _2$ and a matroid $\Delta _2$  such that $\Delta _1 = \G_2 * \D_2$. Proceeding like that we will see  that $\D$ is the join of 1-uniform matroids with possibly a simplex.
\end{proof}

Now we are able to prove Theorem \ref{2-uniform}.

\begin{proof}
Let $\D^*$ denote the simplicial complex with $I_{\D^*} = I(\D)$.
Then the minimal nonfaces of $\D^*$ have dimension 2. Moreover, every vertex of $\D^*$ appears in at least a minimal nonface. Hence the facets of $\D^*$ don't have common vertices. From this it follows that $\D^*$ can not be the join of a complex with a simplex. By Proposition \ref{decomposition 2}, $\D^*$ is a matroid if and only if $\D^*$ is the join of 1-uniform matroids. In this case, the faces of $\D^*$ are unions of faces of the 1-uniform matroids. Hence the minimal nonfaces of $\D^*$ are the minimal nonfaces of  the 1-uniform matroids. The complex generated by the minimal nonfaces of an 1-uniform matroid is just the 2-uniform matroid on the same vertex set.
Therefore, $\D$ is the union of these 2-uniform matroids. Similarly, we can also show that $\D^*$ is the join of 1-uniform matroids if $\D$ is a union of disjoint 2-uniform matroid. Thus, $\D^*$ is a matroid if and only if $\D$ is a union of disjoint 2-uniform matroids. By Theorem \ref{SymCM}, this implies the assertion.
\end{proof}

Replacing condition (iii) of Theorem \ref{2-uniform} by the condition that $\D$ is a union of disjoint $r$-uniform matroids we may expect that the theorem could be extended to arbitrary $r$-dimensional pure complexes. But the following example shows that this is not the case.
 
 \begin{Example}
 {\rm Let $\D$ be the complex generated by 
$\{1,2,3,4\},\{1,2,5,6\},\{3,4,5,6\}.$
Then  
\begin{align*}
I(\D) & = (x_1,x_3) \cap (x_1,x_4) \cap (x_1,x_5) \cap (x_1,x_6) \cap (x_2,x_3) \cap (x_2,x_4)\\
& \quad\; \cap (x_2,x_5) \cap (x_2,x_6) \cap (x_3,x_5) \cap (x_3,x_6) \cap (x_3,x_5) \cap (x_3,x_6).
\end{align*}
 From this it follows that $\D^*$ is generated by the 4-subsets of $\{1,...,6\}$ different than $\{1,2,3,4\},\{1,2,5,6\},\{3,4,5,6\}.$ It is easy to check that $\D^*$ is a matroid.
Hence $I(\D)^{(m)}$ is Cohen-Macaulay for every $m \ge 1$ by Theorem \ref{SymCM}. 
However, $\D$ can not be a union of disjoint 3-uniform matroids.
}
 \end{Example}
 
The \textit{cover  ideal} $J(\Delta)$ of the simplicial complex $\D$ is defined by 
\[
 J(\Delta) = \bigcap_{F \in \mathcal{F}(\Delta)} P_{F},
\]
where $P_F = (x_i|\ i \in F)$. For instance, unmixed squarefree monomial ideals of codimension 2 are just cover ideals of graphs.  \smallskip

The name cover ideal comes from the fact that $J(\D)$ is generated by squarefree monomials $x_{i_1}\cdots x_{i_r}$ with $\{i_1,...,i_r\} \cap F \neq \emptyset$ for every facet $F$ of $\D$, which correspond to covers of the facets of $\D$. Note that the equality between ordinary and symbolic powers of $J(\D)$ were already studied in \cite{HHT}, \cite{HHTZ}.\smallskip

Using a well known result in matroid theory we are able to derive from Theorem \ref{SymCM} the following combinatorial characterization of the Cohen-Macaulayness of large symbolic powers of $J(\D)$.

\begin{Theorem}\label{cover}
Let $\Delta$ be a simplicial complex.
The following conditions are equivalent: \par
{\rm (i)} $J(\Delta)^{(m)}$ is Cohen-Macaulay for every $m\ge 1$.\par
{\rm (ii)} $J(\Delta )^{(m)}$ is Cohen-Macaulay for some $m\ge 3$.\par
{\rm (iii)} $\D$ is a matroid.
\end{Theorem}

\begin{proof}
(i)$\Rightarrow$(ii) is clear.

(ii)$\Rightarrow$(iii) and (iii)$\Rightarrow$(i).
Let $\Delta^c$  be the simplicial complex generated by the complements of the facets of $\D$ in $[n]$.
Then $I_{{\Delta }^c}=J(\Delta )$. 
It is well known that $\D^c$ is a matroid if and only if so is $\D$ \cite[Theorem 39.2]{Sch}.  
Therefore, the assertion follows from Theorem \ref{SymCM}.
\end{proof}

By Lemma \ref{matroid} we can easily test a matroid. For instance, 
if $\dim \D = 1$, $\D$ can be considered as a graph and we obtain the following result.
This result is also proved in [2] by a different method.

\begin{Corollary}
Let $\G$ be a graph without isolated vertices.
Then the following conditions are equivalent:\par
{\rm (i)} $J(\G)^{(m)}$ is Cohen-Macaulay for every $m\ge 1$.\par
{\rm (ii)} $J(\G)^{(m)}$ is Cohen-Macaulay for some $m\ge 3$.\par
{\rm (iii)} Any pair of disjoint edges of $\G$ is contained in a 4-cycle.
\end{Corollary}

\begin{proof}
The assertion follows from the above theorem and Corollary \ref{matroid graph}.
\end{proof}

By Theorem \ref{SymCM} and Theorem \ref{cover}, the Cohen-Macaulayness of  $I_\D^{(m)}$ is equivalent to that of $J(\D)^{(m)}$ for $m \ge 3$.
Therefore, we may ask whether this holds also for $m = 1,2$.
The following example shows that the answer is no.

\begin{Example}
{\rm Let $\D$ be the graph of a 5-cycle. Then $I_\D^{(2)}$ is Cohen-Macaulay \cite[Theorem 2.3]{MiT1} .
We have
$$J(\D) = (x_1,x_2) \cap (x_2,x_3) \cap (x_3,x_4) \cap (x_4,x_5) \cap (x_5,x_1).$$
It can be easily checked that $J(\D)$ is not Cohen-Macaulay so that $J(\D)^{(m)}$  is neither for any $m \ge 2$
by \cite[Theorem 2.6]{HTT}.}
\end{Example}    

\end{document}